\documentclass[reqno, english]{amsart}
\usepackage{etex}
\usepackage{amsmath,amssymb,amsthm,bbm,mathtools,comment}
\usepackage[shortlabels]{enumitem}
\usepackage[pdftex,colorlinks,backref=page,citecolor=blue]{hyperref}
\hypersetup{pdfpagemode=UseNone,pdfstartview={XYZ null null 1.00}}
\usepackage[mathscr]{euscript}
\usepackage[usenames,dvipsnames]{color}
\usepackage{adjustbox,tikz,calc,graphics,babel,standalone}
\usetikzlibrary{shapes.misc,calc,intersections,patterns,decorations.pathreplacing}
\usetikzlibrary{arrows,shapes,positioning}
\usetikzlibrary{decorations.markings}
\usepackage[final]{microtype}
\usepackage[numbers]{natbib}
\usepackage{amsfonts}
\usepackage{caption}
\usepackage{subcaption}
\usepackage{verbatim}
\usepackage{array}
\usepackage[frame,cmtip,arrow,matrix,line,graph,curve]{xy}
\usepackage{pifont}
\usepackage[final]{microtype}
\usepackage{todonotes}
\usepackage{thm-restate}
\usepackage{polski}
\allowdisplaybreaks

\theoremstyle{plain}
\newtheorem{theorem}{Theorem}[section]		
\newtheorem{lemma}[theorem]{Lemma}
\newtheorem{claim}[theorem]{Claim}
\newtheorem{proposition}[theorem]{Proposition}
\newtheorem{corollary}[theorem]{Corollary}

\theoremstyle{remark}
\newtheorem*{remark}{Remark}

\def\N{\mathbb{N}}

\def\FF{\mathcal{F}}
\def\HH{\mathcal{H}}

\def\E{\mathbb{E}}
\renewcommand{\P}{\mathbb{P}}

\DeclareMathOperator{\Ind}{Ind}

\newcommand{\eps}{\ensuremath{\varepsilon}}
\let\emptyset\varnothing

\newcommand{\dD}{d}
\newcommand{\hide}[1]{}

\let\originalleft\left
\let\originalright\right
\renewcommand{\left}{\mathopen{}\mathclose\bgroup\originalleft}
\renewcommand{\right}{\aftergroup\egroup\originalright}

\newcommand{\alex}[1]{{\color{Purple} AS: #1}}

\makeatletter
\def\imod#1{\allowbreak\mkern10mu({\operator@font mod}\,\,#1)}
\makeatother

\author{Xiying Du}
\address[Du]{School of Mathematics, Georgia Institute of Technology, Skiles Building, Cherry Street, Atlanta, Georgia, USA.}
\email{xdu90@gatech.edu}

\author{Ant\'onio Gir\~ao}
\address[Gir\~ao]{Mathematical Institute, University of Oxford, Andrew Wiles Building, Radcliffe Observatory Quarter, Woodstock Road, Oxford, UK.}
\email{girao@maths.ox.ac.uk}

\author{Zach Hunter}
\address[Hunter]{Mathematical Institute, University of Oxford, Andrew Wiles Building, Radcliffe Observatory Quarter, Woodstock Road, Oxford, UK.}
\email{zachary.hunter@exeter.ox.ac.uk}

\author{Rose McCarty}
\address[McCarty]{Department of Mathematics, Princeton University, Fine Hall, Washington Road, Princeton, New Jersey, USA.}
\email{rm1850@princeton.edu}

\author{Alex Scott}
\address[Scott]{Mathematical Institute, University of Oxford, Andrew Wiles Building, Radcliffe Observatory Quarter, Woodstock Road, Oxford, UK.}
\email{scott@maths.ox.ac.uk}

\date{\today} 

\begin{document}
\title{Induced $C_4$-free subgraphs with large average degree}

\thanks{AS and AG were supported by EPSRC grant EP/V007327/1. RM was supported by NSF grant DMS-2202961.}
\begin{abstract}
We prove that there exists a constant $C$ so that, for all $s,k \in \mathbb{N}$, if $G$ has average degree at least $k^{Cs^3}$ and does not contain $K_{s,s}$ as a subgraph then it contains an induced subgraph which is $C_4$-free and has average degree at least $k$. It was known that some function of $s$ and $k$ suffices, but this is the first explicit bound. We give several applications of this result, including short and streamlined proofs of the following two corollaries.

We show that there exists a constant $C$ so that, for all $s,k \in \mathbb{N}$, if $G$ has average degree at least $k^{Cs^3}$ and does not contain $K_{s,s}$ as a subgraph then it contains an induced subdivision of $K_k$. This is the first quantitative improvement on a well-known theorem of K\"uhn and Osthus; their proof gives a bound that is triply exponential in both $k$ and $s$.

We also show that for any hereditary degree-bounded class $\mathcal{F}$, there exists a constant $C=C_\mathcal{F}$ so that $C^{s^3}$ is a degree-bounding function for $\mathcal{F}$. This is the first bound of any type on the rate of growth of such functions. It is open whether there is always a polynomial degree-bounding function.
\end{abstract}
\maketitle

\section{Introduction}

A longstanding conjecture from $1983$ due to Thomassen~\cite{Thomassen} states that for all $g,k \geq 2$, there exists $f(g,k)$ such that every graph $G$ with average degree at least $f(g,k)$ contains a subgraph with girth at least $g$ and average degree at least $k$. It is a standard exercise to show that every graph has a bipartite subgraph with at least half of its edges. So the first nontrivial open case is when $g=5$ and we wish to find a $C_4$-free subgraph in a bipartite graph $G$. (A graph is \textit{$H$-free}, for some graph $H$, if it has no subgraph isomorphic to $H$.) This case of $g=5$ was resolved in a remarkable paper by K\"uhn and Osthus~\cite{KO1} in $2004$. However, since then no further progress has been made, and the conjecture remains wide open for all $g\geq 7$. 

A straightforward probabilistic argument shows that Thomassen's conjecture holds for every ``almost-regular'' graph $G$. Therefore a natural strategy is to pass to an almost-regular subgraph of $G$ which preserves some of its average degree. Unfortunately, this strategy is bound to fail: Pyber, R\"odl, and Szemer\'edi~\cite{PRS} proved that there exist $n$-vertex graphs with average degree $\Omega \left(\log\log n\right)$ which do not contain a $k$-regular subgraph for any $k\geq 3$. We remark that, very recently, Janzer and Sudakov~\cite{JS} proved via an ingenious argument that $\Omega \left(\log\log n\right)$ is indeed the correct barrier. (Formally, they proved that for each $k$ there exists a constant $c_k$ so that every $n$-vertex graph with average degree at least $c_k \log\log n$ has a $k$-regular subgraph.) This fully resolved the Erd\H{o}s-Sauer problem from~\cite{ErdosSauer}.

Since K\"uhn and Osthus~\cite{KO1} first resolved the case of $g=6$, two additional proofs of their theorem have been discovered. The first proof, which is due to Dellamonica, Koubek, Martin, and R\"odl~\cite{DelRodl}, uses a surprising result of F\"uredi~\cite{Furedi} on hypergraphs. The second proof, which is due to Montgomery, Pokrovskiy, and Sudakov~\cite{MPS}, is more recent and gives the best bounds currently known. In particular, they proved that there exists a constant $C$ so that for all $k\in \mathbb{N}$, every graph with average degree at least $k^{Ck^2}$ contains a subgraph which is $C_4$-free and has average degree at least $k$. The same paper also gives a lower bound of $k^{3-o(1)}$, and it is a very interesting problem to determine whether a polynomial in $k$ suffices. 

In $2021$, the fourth author of this paper proved that it is possible to ask for a much stronger property \cite{McCarty}: for all $s,k \in \mathbb{N}$, there exists an integer $f(s,k)$ so that every $K_{s,s}$-free bipartite graph of average degree at least $f(s,k)$ contains an \textit{induced} $C_4$-free subgraph of average degree at least $k$. The paper does not determine an explicit bound; however, with some unwrapping, we believe that the proof yields a function which is triply exponential in $s$, for fixed $k$. Moreover, the assumption about bipartiteness can be removed, at the cost of another exponent, using a theorem of Kwan, Letzter, Sudakov, and Tran~\cite{KLST}. Their theorem says that for all $s,k \in \mathbb{N}$, there exists an integer $g(s,k)$ so that every $K_s$-free graph with average degree at least $g(s,k)$ contains an \textit{induced} bipartite subgraph with average degree at least $k$. (More precisely, they proved that there exists a constant $C$ so that the function $g(s,k) = \exp \left( C^{s \log{s}} \cdot k \log{k}\right)$ suffices.)

In our main theorem, we prove that this stronger induced theorem for $K_{s,s}$-free graphs holds with bounds that essentially match the singly exponential bounds of~\cite{MPS} from the non-induced setting. 

\begin{theorem}\label{thm:main}
There exists a constant $C$ so that for all $s,k \in \mathbb{N}$, every $K_{s,s}$-free graph with average degree at least $k^{Cs^3}$ contains an induced subgraph which is $C_4$-free and has average degree at least $k$.
\end{theorem}

First, we note that the assumption of $K_{s,s}$-freeness is essential, as any $C_4$-free induced subgraph of $K_{s,s}$ has average degree less than~$2$ (thus by taking $G= K_{s,s}$ for larger and larger $s$, one could make the average degree of $G$ arbitrarily large without obtaining induced $C_4$-free subgraphs with average degree at least~$2$). Moreover, we remark that if a graph $G$ contains a $K_{4k^2,4k^2}$-subgraph, then it contains a bipartite $C_4$-free subgraph with average degree at least $k$; see~\cite[Lemma~2.2]{MPS}. (In fact, we can just take this subgraph to be the incidence graph of the projective plane of order $q$, where $q$ is the smallest prime which is at least $k-1$.) Therefore we obtain the following as a simple corollary of Theorem~\ref{thm:main}.

\begin{corollary}
\label{cor:MPS}
There exists a constant $C$ so that for all $k \in \N$, every graph with average degree at least $k^{Ck^6}$ contains a bipartite $C_4$-free subgraph with average degree at least $k$.
\end{corollary}

\noindent Corollary~\ref{cor:MPS} nearly recovers the result of \cite{MPS}, only with the exponent `$k^2$' replaced by `$k^6$'. 

We also prove some corresponding lower bounds for Theorem~\ref{thm:main}. So, let $f_{\Ind}(s,k)$ denote the least integer $D$ such that if $G$ is a $K_{s,s}$-free graph with average degree at least $D$, then $G$ contains an \textit{induced} $C_4$-free subgraph with average degree at least $k$. Theorem~\ref{thm:main} establishes that $f_{\Ind}(s,k)\le k^{Cs^3}$ for some absolute constant $C$. By considering random graphs, one can obtain some lower bounds for this quantity as well. Specifically, in Section~\ref{sec:lower} we prove the following. 

\begin{theorem}\label{thm:lower} The following bounds hold.
    \begin{enumerate}[label=(\alph*)]
        \item\label{diagonal case} There exists a constant $c>0$ so that for all $k\ge 2$, we have that $f_{\Ind}(k,k)\ge k^{ck}$.
        \item\label{fixed k case} For each fixed $k\ge 2$, we have that $f_{\Ind}(s,k)\ge s^{(1/4-o(1))(k^2-3k-2)}$.
        \item\label{fixed s case} For each fixed $s\ge 2$, we have that $f_{\Ind}(s,k)\ge k^{(1/2-o(1))s-1}$.
    \end{enumerate}
\end{theorem}

\subsection{Further applications} 
Now we discuss two additional corollaries of Theorem~\ref{thm:main} which are motivated by $\chi$-boundedness. A family of graphs $\mathcal{F}$ is \textit{hereditary} if for every $G\in \mathcal{F}$, every induced subgraph of $G$ is also in $\mathcal{F}$. A hereditary family of graphs $\mathcal{F}$ is \textit{$\chi$-bounded} if there exists a function $f: \mathbb{N} \rightarrow \mathbb{N}$ such that for every $G\in \mathcal{F}$, we have $\chi(G) \leq f(\omega(G))$, where we write $\chi(G)$ for the chromatic number of $G$ and $\omega(G)$ for the clique number of $G$. 
These classes have been widely studied, following an influential paper of Gy\'arf\'as~\cite{G87} raised a number of well-known conjectures (many of these have now been solved: see, for example,  \cite{CSS, inducedx}; and \cite{SSSurvey} for a survey). 

A similar notion to $\chi$-boundedness, but for average degree instead of chromatic number, has recently been receiving attention. Intuitively, a class is $\chi$-bounded if large chromatic number forces the existence of large cliques; and it is ``degree-bounded'' if large average degree forces large balanced bicliques. Formally, we say that a hereditary family $\mathcal{F}$ is \textit{degree-bounded} if there exists a function $g: \mathbb{N} \rightarrow \mathbb{N}$ such that for every $G\in \mathcal{F}$, we have $d(G) \leq g(\tau(G))$, where we write $d(G)$ for the average degree of $G$ and $\tau(G)$ for the \textit{biclique number} of $G$, which is the largest integer $s$ such that $G$ contains a (not necessarily induced) copy of $K_{s,s}$. Any such function $g$ is called a \textit{degree-bounding function} for the class. 

Strengthening an unpublished theorem of Hajnal and R\"{o}dl (see~\cite{GST80}), Kierstead and Penrice~\cite{KP} showed that for every tree $T$, the class of graphs without an induced copy of $T$ is degree-bounded. Recently, Scott, Seymour and Spirkl~\cite{SSS} established a quantitative strengthening of the result from \cite{KP}, by showing that the class of $T$-induced-free graphs is polynomially degree bounded (that is, the function $g$ can be a polynomial in $s$, for fixed $T$). This improved upon another recent theorem of Bonamy, Bousquet, Pilipczuk, Rz\k{a}\.zewski,
Thomass\'e and Walczak \cite{bonamy}, who proved the same result when $T$ is a path. 

A rephrasing of Theorem~\ref{thm:main} immediately gives the following corollary, which says that every hereditary degree-bounded class is essentially exponentially bounded. 

\begin{corollary}\label{cor:boundedness}
For any hereditary degree-bounded class of graphs $\mathcal{F}$, there exists a constant $C=C_{\mathcal{F}}$ so that $C^{s^3}$ is a degree-bounding function for $\mathcal{F}$.
\end{corollary}

This corollary is surprising since the analogous statement for $\chi$-boundedness is false: Bria\'nski, Davies, and Walczak~\cite{non-polychi} recently constructed hereditary $\chi$-bounded classes where the optimal $\chi$-bounding function grows arbitrarily quickly. This disproved the popular conjecture of Esperet~\cite{esperet2017habilitation} that every hereditary $\chi$-bounded class is polynomially $\chi$-bounded. It could still be true, however, that every hereditary degree-bounded class is polynomially degree-bounded. We note that, informally, Theorem~\ref{thm:main} also tells us that the value of $g(1)$ can be used to ``efficiently control'' the rate of growth of $g$. Again the analogous statement for $\chi$-boundedness is very false: Carbonero, Hompe, Moore, and Spirkl~\cite{CHMS23} recently proved that there exist graphs of arbitrarily large chromatic number where every triangle-free induced subgraph has chromatic number at most~$4$ (for more general results see \cite{inducedinduced}).

Our final corollary obtains improved bounds for classes with a forbidden induced subdivision. A \textit{subdivision} of a graph $H$ is any graph which can be obtained from $H$ by replacing the edges $uv$ of $H$ by internally-disjoint paths between $u$ and $v$. A beautiful result of K\"uhn and Osthus~\cite{KO1} states that for every graph $H$, the class of graphs which do not contain an induced subdivision of $H$ is degree-bounded. 

\begin{theorem}[K\"uhn and Osthus]\label{kotheorem}
    For every graph $H$ and integer $s$, there is an integer $p(s,H)$ such that every $K_{s,s}$-free graph $G$ with average degree at least $p(s,H)$ contains an induced subdivision of $H$. 
\end{theorem}
Their bounds for $p(s,H)$ are roughly triply exponential in $s$, for fixed $H$. A conjecture raised by Bonamy et al.~\cite[Conjecture~33]{bonamy} claims that actually $p(s,H)$ could be taken to be a polynomial in $s$. In Section~\ref{finding subdivisions}, we give a short and streamlined proof that a single exponential is enough.

\begin{restatable}{corollary}{indSub}
\label{induced subdivision}
    There exists a constant $C$ so that for all $k,s \in \N$, every $K_{s,s}$-free graph $G$ with average degree at least $k^{Cs^3}$ contains an induced subdivision of $K_k$. 
\end{restatable}

 Our proof will use Theorem~\ref{thm:technical} (which is a technical strengthening of our main theorem, Theorem~\ref{thm:main}) and a theorem about (non-induced) subdivisions which was proved independently by Bollob\'as and Thomason~\cite{bollobas} and K\'omlos and Szemer\'edi~\cite{komlos}.

The rest of this paper is organized as follows: after giving some notation in the next section, we prove some preliminary lemmas in sections \ref{sec:prelems} and \ref{hyperproblem}.  Theorem \ref{thm:main} is proved in Section \ref{sec:mainproof}, Theorem \ref{thm:lower} in Section \ref{sec:lower}, and Corollary \ref{induced subdivision} in Section \ref{finding subdivisions}.  We conclude with some discussion in section \ref{sec:conclude}.

\section{Notation}

We mostly use standard notation, and we consider all graphs to be finite, simple, and loopless. Let $G$ be a graph. We denote the vertex set of $G$ by $V(G)$ and the edge set of $G$ by $E(G)$. We write $|G| \coloneqq |V(G)|$ and $e(G)\coloneqq |E(G)|$. We write $d(G)$ for the average degree of $G$ (so that $d(G)= 2e(G)/|G|$), $\delta(G)$ for the minimum degree of $G$, and $\Delta(G)$ for the maximum degree of $G$. Given a vertex $v$ of $G$, we write $d_G(v)$ for the degree of $v$ and $N_G(v)$ for the neighbourhood of $v$.

As usual, we say that a graph $H$ is an \textit{induced subgraph} of $G$ if there is an injective map $f:V(H) \rightarrow V(G)$ so that $f(xy)\in E(G)$ if and only if $xy\in E(H)$. We say that $G$ is \textit{$d$-degenerate} if every non-empty subgraph $G'$ of $G$ contains a vertex of degree at most $d$. We say that $G$ is \textit{$H$-free} if $G$ does not contain $H$ as a subgraph. We say that $G$ is \textit{$H$-induced-free} if $G$ does not contain $H$ as an induced subgraph. 

Finally, given two disjoint sets of vertices $A,B\subset V(G)$, we write $G[A,B]$ for the induced bipartite graph with parts $A$ and $B$. That is, the vertex set of $G[A,B]$ is $A \cup B$, and the edge set of $G[A,B]$ is $\{xy \in E(G): x \in A, y \in B\}$. We view bipartite graphs as being equipped with a fixed bipartition; typically we denote a bipartite graph $G$ with sides $A$ and $B$ and edge set $E$ by $G=(A,B,E)$.

\section{Preliminary lemmas}\label{sec:prelems}
Recall from the introduction that a straightforward probabilistic argument shows that Thomassen's Conjecture holds for graphs which are almost regular; this is done by including edges independently at random and deleting one edge from each short cycle. It turns out that we can similarly use the K\H{o}v\'ari-S\'os-Tur\'an Theorem~\cite{kovari} to prove Theorem~\ref{thm:main} for graphs which are almost regular, by including vertices independently at random and deleting one vertex from each short cycle. We will take care of this step in Lemma~\ref{deletion}.


For now, we recall the classical result of K\H{o}v\'ari, S\'os, and Tur\'an. (We note that this can be formulated either for graphs or for bipartite graphs.  There is a standard reduction to the bipartite case: for a $K_{s,s}$-free graph $G$, we create a new graph with two vertices $u_1$ and $u_2$ for each vertex $u$ of $G$, where $u_i$ and $v_j$ are adjacent in the new graph if $i \neq j$ and $uv$ is an edge of $G$.)


\begin{theorem}[{\cite{kovari}}]\label{edges in Kssfree}
For all $s\ge 2$ and $n\ge  1$, every $K_{s,s}$-free graph on $n$ vertices has $O((s-1)^{1/s}n^{2-1/s}+sn ))$ edges. Thus, there exists an absolute constant $C$ such that for all $s \ge 2$ and $n\ge s^2$, every $K_{s,s}$-free graph on $n$ vertices has at most $Cn^{2-1/s}$ edges.
\end{theorem}

Since we will be able to take care of the nearly regular case easily, it is helpful to have configurations which guarantee a nearly-regular induced subgraph of large average degree. Our first lemma says that it is enough to find a bipartite graph of large average degree where each side \textit{individually} is nearly regular (but the average degrees of the two sides can be quite different from each other). This lemma is due to Janzer and Sudakov~{\cite[Lemma~3.7]{JS}}, but we restate the proof since we will need to state a slightly stronger conclusion. We say that a bipartite graph $\Gamma = (A,B,E)$ is {\em $L$-almost-biregular} if $d_\Gamma(a)\le L|E|/|A|$ for all $a\in A$,  and $d_\Gamma(b)\le L|E|/|B|$ for all $b\in B$. Thus, informally, each vertex in $A$ has degree at most a factor of $L$ larger than the average degree of all vertices in $A$, and likewise for $B$. 

\begin{lemma}\label{biregular to regular} Each $L$-almost-biregular graph $\Gamma$ has an induced subgraph $\Gamma'$ with $d(\Gamma')\ge d(\Gamma)/4$ and $\Delta(\Gamma')\le 24Ld(\Gamma')$.
\end{lemma}
\begin{proof}
    Suppose $\Gamma = (A,B,E)$. If $E = \emptyset$, there is nothing to prove (we may take $\Gamma' = \Gamma)$. 
    Otherwise, as both $A$ and $B$ have vertices of nonzero degree, we have
    $L\frac{|E|}{|A|},L\frac{|E|}{|B|}\ge 1$.
    
    We may assume $|A|\le |B|$. Let $p = |A|/|B|$. We take a random subset $B'\subset B$ by adding each $b\in B$ independently with probability $p$. Let $A'\subset A$ be the set of $a\in A$ such that $|N_\Gamma(v)\cap B'|\le 1+2p(d_\Gamma(v)-1)$. We shall take $\Gamma' = \Gamma[A',B']$, and show this works with positive probability.

    By construction, we have $\Delta(\Gamma')\le 1+2L\frac{|E|}{|B|}\le 3L\frac{|E|}{|B|}$. Indeed, each vertex $a\in A'$ has degree at most $1+2pL|E|/|A|\le 3L |E|/|B|$, and each vertex $b\in B'$ has degree at most $d_\Gamma(b)\le L|E|/|B|$.

    Consider any $e= ab\in E$, and let $U = (N_\Gamma(a)\setminus \{b\}) \cap B'$. Applying Markov's inequality, we see that \[\mathbb{P}(a\in A'|b\in B') = \mathbb{P}(|U| \le 2\E[|U|])>   1/2.\] 
    Then
    $\E[|A'|+|B'|] \le |A|+\E[|B'|]=2|A|$.
    and 
    \[\E[e(\Gamma')] = \sum_{e=ab\in E} \mathbb{P}(b\in B')\mathbb{P}(a\in A'|b\in B') > p|E|/2.\]
    Thus we have 
    \begin{align*}
        \E[4e(\Gamma') - \frac{|E|}{|B|}(|A'|+|B'|)]
        &> 4(p|E|/2) - \frac{|E|}{|B|}2|A| =0.
    \end{align*}Consequently, we can choose $A',B'$ such that the expectation is positive.  Then $\Gamma'$ is non-empty and, dividing by $|\Gamma'|=|A'|+|B'|$, we get
    \[
    d(\Gamma')\ge 2\frac{|E|}{4|B|} \ge d(\Gamma)/4.
    \] 
    Since $\Delta(\Gamma')\le 3L|E|/|B|$,
    \[
    12Le(\Gamma')\ge \Delta(\Gamma')(|A'|+|B'|) 
    \]
    which implies that $\Delta(\Gamma')\le 24L d(\Gamma')$ 
    as desired.
\end{proof}




We can now use Theorem~\ref{edges in Kssfree} to deal with the case of nearly regular graphs, as discussed.

\begin{lemma}\label{deletion}

There exists an absolute constant $\kappa>0$ so that the following holds for all $s\ge 2$, $\delta\in (0,1/10)$ and $\dD \ge \kappa^{s/\delta}$. 

Let $G$ be a $K_{s,s}$-free graph with $\Delta(G)\le d$ and average degree $d(G)\ge \dD^{1-\delta /s}$. Then, $G$ contains an induced $\{C_3,C_4\}$-free subgraph $H$ with average degree at least $\dD^{(1-10\delta)/5s}$ and $\Delta(H)\le d(H)^{1/(1-10\delta)}$.

\end{lemma}
\begin{remark}
    We will need the upper bound on $\Delta(H)$ to derive a more technical version of Theorem~\ref{thm:main} (namely Theorem~\ref{thm:technical}), which will be useful in Section~\ref{finding subdivisions} and future applications.
\end{remark}
\begin{proof}Let $\mathcal{C}_4$ be the set of $4$-cycles $C$ in $G$. Let $\mathcal{S}_4$ be the set of pairs $(e,C)\in E(G)\times \mathcal{C}_4$ such that 
$e$ contains exactly one vertex of $C$ (and so $e\cup E(C)$ spans 5 vertices).

By K\H{o}v\'ari-S\'os-Tur\'an (Theorem~\ref{edges in Kssfree}), for each edge $xy$ there are at most $O(\dD^{2-1/s})$ edges between $N(x)$ and $N(y)$ and so at most $O(\dD^{2-1/s})$ 4-cycles passing through $xy$. As there are at most $nd$ edges, we have $|\mathcal{C}_4| = O(n\dD^{3-1/s})$. Thus $|\mathcal{S}_4| \le 4\dD |\mathcal{C}_4|=O(n\dD^{4-1/s})$.

Similarly, let $\mathcal{C}_3$ be the set of $3$-cycles $C\subset G$ and $\mathcal{S}_3$ be the set of pairs $(e,C)\in E(G)\times \mathcal{C}$ such that $e$ contains exactly one vertex of $C$.  By K\H{o}v\'ari-S\'os-Tur\'an, for each vertex $x$ there are at most $O(d^{2-1/s})$ edges in $N(x)$, and so each vertex belongs to at most $O(d^{2-1/s})$ triangles.  We conclude that $|\mathcal{C}_3|=O(nd^{2-1/s})$, and $|\mathcal{S}_3|\le 3\dD |\mathcal{C}_3|=O(nd^{3-1/s})$.

We will take $p = \dD^{(1/5s)-1}$. Let $U$ be a random subset of $V(G)$ with each vertex independently included with probability $p$. Let $\mathcal{C}'\subset \mathcal{C}_3 \cup \mathcal{C}_4$ be the set of $C$ such that $V(C)\subset U$, and define \[U' :=  \bigcup_{C\in \mathcal{C}'}V(C)\]
and
\[U^* := \{u\in U: |N_G(u)\cap U|\ge 1+4pd_G(u)\}.\]We shall consider $H:= G[U-U'-U^*]$, which clearly lacks $3$-cycles and $4$-cycles, and also has $\Delta(H)\le 1+4p\dD\le 5\dD^{1/5s}$.

Let 
\begin{align*}
X &= e(G[U]),\\
X^* &= \sum_{u\in U^*}|N_G(u)\cap U|, \\
Y&= |\mathcal{C}'|,\\
\intertext{and}
Z&= \#\{(e,C)\in \mathcal{S}_3\cup \mathcal{S}_4: e\in E(G[U])
\textrm{ and } C\in \mathcal{C}'\}.
\end{align*}

We observe that $e(H) \ge X-X^*-6Y-Z$ (indeed, for each $u\in U^*$, we lose at most $|N_G(u)\cap U|$ edges by deleting $u$; and for each $C\in \mathcal{C}'$, and any set $U\subset V(G)$, we have that $e(G[U])-e(G[U-V(C)]) \le \binom{|V(C)|}{2}+e(V(C),U\setminus V(C))$
(here the first term counts internal edges, and the second term counts the rest).

Now 
\[\E[X] = p^2e(G) \ge \frac{1}{2}n\dD^{1-\delta/s}p^2 
=\frac12npd^{(1/5 -\delta)/s}.\]
Applying Markov's inequality (as in Lemma~\ref{biregular to regular}), we get
\begin{align*}
\E[X^*] &= \sum_{u\in V(G)}\sum_{v\in N_G(u)} p^2\mathbb{P}(|N_G(u)\cap U|\ge 1+4pd_G(u)|\{u,v\}\subset U)\\
&\le \sum_{u\in V(G)} d(u) p^2/4\\
&\le p^2e(G)/2.
\end{align*}
Also, using our upper bounds on $|\mathcal{C}_3|$ and $|\mathcal{C}_4|$,
\[
\E[Y] = p^4|\mathcal{C}_4|+p^3|\mathcal{C}_3| = O(n\dD^{-1-1/5s})
=O(npd^{-2/5s}),
\]
and
\[
\E[Z] = p^5|\mathcal{S}_4|+p^4|\mathcal{S}_3|= O(n\dD^{-1})
=O(npd^{-1/5s}).
\]

Now, using the assumption that $\dD$ is sufficiently large with respect to $s,\delta$, we get
\[
\E[X-X^*-5\dD^{(1/5-2\delta)/s}|U|-6Y-Z] 
\ge np\left(\dD^{(1/5-\delta)/s}/4-5\dD^{(1/5-2\delta)/s}-O(\dD^{-1/5s})\right) \ge 0.
\]
Choosing $U$ such that the above holds, we have that $e(G[U\setminus (U'\setminus U^*)]) \ge 5\dD^{(1/5-2\delta)/s}|U|$ so $H = G[U\setminus (U'\cup U^*)]$ has average degree at least $\dD' :=5\dD^{(1/5-2\delta)/s}$ and is $C_4$-free with $\Delta(H)\le 5\dD^{1/5s}\le \dD'^{1/(1-10\delta)}$, as desired.\end{proof}

We now need the following lemma.  Roughly speaking, it states that every graph $G$ with sufficiently large average degree either contains an induced subgraph $G'\subseteq G$ which is almost-regular with still high average degree or we can find a very unbalanced bipartite (not necessarily induced) subgraph which still preserves many edges. 

\begin{lemma}\label{make extreme} There exists an absolute constant $\kappa$ so that the following holds for all $\dD\ge 2$ and $\delta\in (0,1/5)$ with $\dD\ge (1/\delta)^{\kappa /\delta}$.

Let $G_0$ be an $n$-vertex graph with $d(G_0)=\dD$ which is $\dD$-degenerate. Then either $G_0$ contains an induced subgraph $G^*$ of average degree $d(G^*)\geq 6\dD^{1-5\delta}$ and maximum degree $\Delta(G^*)\le 6\dD^{1+3\delta}$, or we can find a partition $V(G_0)=A\cup B$ such that $e(G_0[A,B])\geq n\dD/4$ and $|B|\leq 4|A|/2^{\dD^\delta}$.  
\end{lemma}

\begin{remark}\label{about kappa}
    In the proofs of Lemma~\ref{make extreme} and Lemma~\ref{make bipartite} below, we include the size assumption $\dD\ge x^{\kappa x}$ (respectively with $x=1/\delta$ and $x=s/\delta$) so that various inequalities of the form ``$C_1\log \dD\le \dD^{1/C_2x}$'' hold. 
    These will be useful later, and also imply
    that $\dD^\delta$ is at least a large constant.
    
    \hide{In the proofs of Lemma~\ref{make extreme} and Lemma~\ref{make bipartite} below, we need the assumption $\dD\ge x^{\kappa x}$ (respectively with $x=1/\delta$ and $x=s/\delta$) so that various inequalities of the form ``$C_1\log \dD\le \dD^{1/C_2x}$'' hold (for example, it will be important that $d^{\delta}$ is at least a large constant). Going forward, we shall tacitly use such inequalities in proofs without repeatedly referencing this largeness assumption (though these largeness assumptions will be included in the statements of the results).}
    
    \hide{The size assumption on $d$ in 
    the proofs of Lemma~\ref{make extreme} and Lemma~\ref{make bipartite} below is needed for various inequalities of the form ``$C_1\log \dD\le \dD^{1/C_2x}$''; it will also be important that $d^{\delta}$ is at least a large constant. Similar inequalities will be needed in later proofs.}
    
\end{remark}
\begin{proof}Let $V= V(G_0)$, and let $R\subset V$ be the vertices with degree greater than $\dD2^{\dD^{\delta}}$. Note that $e(G_0)\le nd$ and so $|R|< n2^{1-\dD^{\delta}}$. If $e(G_0[R,V\setminus R])\geq n\dD/4$ then we can simply take $B= R, A = V\setminus R$.
So we may assume that $e(G_0[R,V\setminus R])\leq n\dD/4$. 

Since $G_0[R]$ is $\dD$-degenerate (and we may assume that $\kappa$ is not too small), $e(G_0[R]) \le \dD|R|< n\dD/8$, and so $e(G_0[V\setminus R])\geq n\dD/8$. Let $G$ be an induced subgraph of $G_0[V\setminus R]$ with maximal average degree: so $d(G)\ge \dD/4$ and $\delta(G)\ge d(G)/2 \ge \dD/8$. Note that $\Delta(G)\le \dD2^{\dD^\delta}$ as it contains no vertices from $R$. 

We split $V(G)$ dyadically, according to the degree of the vertices. For $-3\le i \le \dD^{\delta}$, 
take $C_i\coloneqq \{x\in V(G): 2^i \dD\leq d_{G}(x)<2^{i+1}\dD\}$; set $\dD_i := 2^i \dD$ 
and $E_i:= \sum_{v\in C_i} d_{G}(v)$. By our bounds on $\delta(G),\Delta(G)$, the sets $C_{-3},C_{-2},\dots,C_{\dD^{\delta}}$ partition $V(G)$.

By the pigeonhole principle, we can find $j$ so that $E_j\ge 2e(G)/\dD^{2\delta}$. We now take a random subset $C'\subset C_j$ where each vertex $v\in C_j$ is included independently with probability $1/4$. Let $C'' \subset C'$ be the set of $v\in C'$ such that $|N_G(v)\cap C'|\le d_G(v)/2$. By Markov's inequality, for $v\in C_j$, we have that \[\mathbb{P}(v\in C''|v\in C') \ge \mathbb{P}(|N_G(v)\cap C'|\le 2\E[|N_G(v)\cap C'|])\ge 1/2.\] Consequently, we have
\[\E[|C''|] =\sum_{v\in C_j}\mathbb{P}(v\in C')\mathbb{P}(v\in C''|v\in C')\ge |C_j|/8.\]

Fix some choice of $C'$ such that $|C''|\ge|C_j|/8$. Note that for all $v\in C''$ and any $S\subset C''$, we have that $|N_G(v)\setminus S|\in [\dD_j/2,2\dD_j]$ (the upper bound by the definition of $C_j$ and the inclusion $C''\subset C_j$, and the lower bound by the definition of $C''$ and the inclusion $C''\subset C'$). Since $G[C'']$ is $\dD$-degenerate, we have $e(G[C'']) \le \dD|C''|$. Thus $R':=\{v\in C'': |N_G(v)\cap C''| \ge 4 \dD\}$ satisfies $|R'| \le |C''|/2$. Let $C''':= C''\setminus R'$.

It follows that $|C'''|\ge |C''|/2$, and thus $e(G[C''',V(G)\setminus C'''])\ge |C'''|\dD_j/2 \ge \dD_j|C_j|/32\ge E_j/64$.

Now consider: $V(G) \setminus C'''$: we partition these dyadically, now according to neighbours in $C'''$. Specifically, for $-\log \dD \le i \le \dD^{\delta}$, let $C_i^{*}:= \{z\in V(G)\setminus C''': 2^i \dD\le |N_G(z)\cap C'| <2^{i+1}\dD\}$ and $E_i^*:= \sum_{v\in C_i^*} |N_G(v)\cap C'''|$. The sets $C_i^*$ partition $V(G)\setminus (C''' \cup \{v:|N_G(v)\cap C'''|=0\})$.

By pigeonhole, there exists some $k$ such that $e(G[C''',C_k^*]) \ge e(G[C''',V(G)\setminus C'''])/\dD^{2\delta}$. For $v\in C_k^*$, we have $|N_G(v)\cap C'''|\in [2^k\dD,2^{k+1}\dD]$. Let $R^* := \{v\in C_k^*: |N_G(v)\cap C_k^*|\ge 4\dD\}$. By $\dD$-degeneracy, we have that $|R^*|\le |C_k^*|/2$. Thus, taking $C^{**}:= C_k^*\setminus R^*$, we have that $e(G[C''',C^{**}])\ge e(G[C''',C_k^*])/4\ge e(G)/128\dD^{4\delta}$.

For convenience, write $A:= C''', B:= C_k^{**}$ and consider $\Gamma := G[A,B]$. Since $\Delta(G[A]),\Delta(G[B]) \le 4\dD$, we have $\Delta(G[S])\le \Delta(\Gamma[S])+4\dD$ for every $S\subset A\cup B$. Thus we will be done by finding $S\subset A\cup B$ such that $d(\Gamma[S])\ge 6\dD^{1-5\delta}$ and $\Delta(\Gamma[S])\le \dD^{3\delta}d(\Gamma[S])\le 2\dD^{1+3\delta}$ (note that $G_0\supset \Gamma$ being $\dD$-degenerate implies $d(\Gamma[S])\le 2\dD$).

We now show that $\Gamma$ is $16\dD^{2\delta}$-almost-biregular.  Since $d(\Gamma)\ge \dD^{1-4\delta}/128$, we are then done by an application of Lemma~\ref{biregular to regular}. To verify almost-biregularity, we first note \[\max_{b\in B}\left\{\frac{d_\Gamma(b)|B|}{e(\Gamma)}\right\}\le \frac{\max_{b\in B}\{d_\Gamma(b)\}}{\min_{b\in B} \{d_\Gamma(b)\}} \le 2\](where the last inequality is because $d_\Gamma(b)=|N_G(b)\cap C'''|\in [2^k\dD,2^{k+1}\dD]$ for all $b\in B$). Meanwhile, writing $\Gamma' := G[A,V(G)\setminus A]$ we similarly have that \[\max_{a\in A}\left\{\frac{d_{\Gamma'}(a)|A|}{e(\Gamma')}\right\}\le \frac{\max_{a\in A}\{d_{\Gamma'}(a)\}}{\min_{a\in A} \{d_{\Gamma'}(a)\}} \le 4,\] and so $\max_{a\in A}\left\{\frac{d_{\Gamma}(a)|A|}{e(\Gamma)}\right\}\le \max_{a\in A}\left\{\frac{d_{\Gamma'}(a)|A|}{e(\Gamma)}\right\}\le 4\frac{e(\Gamma')}{e(\Gamma)}\le 16\dD^{2\delta}$.\end{proof}

\begin{lemma}\label{make bipartite}
There exists an absolute constant $\kappa$ so that the following holds for all $s\ge 2$, $\delta\in (0,1/2)$ and $\dD\ge (s/\delta)^{\kappa s/\delta}$.

Let $G_0$ be a $\dD$-degenerate, $K_{s,s}$-free graph with $V(G_0)=A_0\cup B$. Suppose that $|A_0|\geq 2^{\dD^{\delta}}|B|$ and $e(G_0[A_0,B])\ge 3\dD^{1/2}|A_0|$. 

Then, for any $r\le \dD^{\min\{\delta/4,1/3s\}}$, we can find subsets $A'\subset A_0, B'\subset B$ such that $G_0[A'],G_0[B']$ are independent sets, $|A'|\geq 2^{\dD^{\delta/2}}|B'|$, and $|N(a)\cap B'| = r$ for $a\in A'$.
\end{lemma}
\begin{remark}
    In our applications, we will only need the special case where $e(G_0[A_0,B])\ge \dD|A_0|/10^4$.
\end{remark}

\begin{proof}First, we clean $G_0$ so that our larger part has reasonably bounded degree. Let $A^*:= \{a\in A_0: d_{G_0}(a)\ge 10\dD\}$. We have have that $e(A^*,B)\ge 10\dD|A^*|$ and since $G_0$ is $\dD$-degenerate we must also have $e(A^*,B)\le \dD(|A^*|+|B|)$, implying $|A^*|\le |B|$ and $e(A^*,B)\le 2\dD|B|$. Next let $A^{**}:=\{a\in A_0:d_{G_0}(a)\le \dD^{1/2}\}$. We have that $e(A^{**},B)\le \dD^{1/2}|A_0|$.

Let $A_1:= A_0\setminus (A^*\cup A^{**})$. Then $e(G_0[A_1,B]) \ge (3\dD^{1/2} - 2\dD|B|/|A_0|-\dD^{1/2})|A_0|\ge \dD^{1/2}|A_0|$ and so $|A_1|\ge e(G_0[A_1,B])/10\dD \ge |A_0|/10\dD^{1/2}$. Note that $d_{G_0}(a)\in [\dD^{1/2},10\dD]$ for $a\in A_1$. 

As $G_0$ is $\dD$-degenerate, it has chromatic number at most $\dD+1$. So we can find $A\subset A_1$ with $|A|\ge |A_1|/2\dD \ge |A_0|/20\dD^{3/2}\ge |B|2^{2d^{\delta/2}}$, such that $G_0[A]$ is an independent set. Let $G = G_0[A\cup B]$.

Now let $D$ be an orientation of $G[B]$ such that $|N_D^+(b)|\le \dD$ for each $b\in B$ (this can be done as $G$ is $\dD$-degenerate). Let $B_0' \subset B$ be a random subset where each $b\in B$ is kept with probability $p = 1/\dD^2$. Let $B'$ be the set of $b\in B_0'$ where $|N_D^+(b)\cap B_0'|= 0$.

Consider any $a\in A$. Since $G[N_G(a)]$ is $K_{s,s}$-free and has at least $\dD^{1/2}$ vertices, we can fix an independent set $I\subset N_G(a)$ of size $r$ (as $r\le \dD^{1/3s}$) by K\H{o}v\'ari-S\'os-Tur\'an (Theorem~\ref{edges in Kssfree}). Let $\mathcal{E}_a$ be the event that $N_G(a)\cap B' = N_G(a)\cap B_0' = I$. Let $X = (N(a)\setminus I) \cup \bigcup_{b\in I} N_D^+(b)$. Then $|X|\le \dD(10+r)\le \dD^2/2$. Consequently, \[\mathbb{P}(\mathcal{E}_a) = p^{|I|}(1-p)^{|X|} \ge 2^{-2r\log \dD} (1/2)\ge 2^{-\dD^{\delta/3}}.\]

We set $A':= \{a\in A:\mathcal{E}_a\textrm{ holds}\}$. We have that $\E[|A'|] \ge |A|2^{-\dD^{\delta/3}} \ge 2^{\dD^{\delta/2}}|B| \ge 2^{\dD^{\delta/2}}|B'|$. Thus, there is some choice of $A',B'$ where this inequality holds. Fixing such a choice, we are done.
\end{proof}








We will also need the following slight variant of a result of F\"uredi ({\cite[Theorem~$1'$]{Furedi}}; see also \cite{d11}). We say a hypergraph $\mathcal{H}$ is {\em $s$-bounded} if $|e|\le s$ for each $e\in E(\mathcal{H})$.


\begin{theorem}\label{kernels}Let $s,r,t$ be positive integers, and let $T= \sum_{k=0}^s \binom{r}{k}$.
Let $\FF = (V,E)$ be an $r$-uniform hypergraph. There exist $E^* \subset E$ with $|E^*| \ge (tr2^{2r+1})^{-T-1}|E|$ and $c:V\to[r]$ such that:
\begin{itemize}
    \item $c(F) =[r]$ for all $F\in E^*$ (i.e.~$\FF^* := (V,E^*)$ is $r$-partite with vertex classes given by the sets $c^{-1}(i)$);
    \item there is an $s$-bounded hypergraph $\HH$ on vertex set $[r]$, such that, for $e\in \binom{[r]}{\le s}$:
    \begin{itemize}
    \item if $e\in E(\HH)$, then for every $F\in E^*$, there are distinct $F_1,\dots,F_t\in E^*$ such that $c(F\cap F_i)\supset e$ for each $i\in [t]$;
    \item if $e\not\in E(\HH)$, then for all distinct $F,F'\in E^*$, we have $c(F\cap F')\not\supset e$.
    \end{itemize}
    Thus if $F_0$ is contained in an edge of $E^*$, then $F_0$ extends to an edge of $\FF^*$ in at least $t$ different ways if $c(F_0)$ is an edge of $\HH$, and otherwise it has a unique extension.
\end{itemize}
\end{theorem}

\begin{proof}Considering a random coloring $c:V\to [r]$, we have that $\mathbb{P}(c(F) =[r]) = \frac{r!}{r^r} \ge \exp(-r)$ for each $F\in E$. Thus, we can choose some $c$ such that $E_0:= \{F\in E:c(F) = [r]\}$ satisfies $|E_0|\ge \exp(-r)|E|$.

We run an iterative cleaning process, giving $E_0\supset E_1 \supset \dots \supset E_\tau$ (where $\supset$ denotes non-strict containment). For $i =0,\dots, \tau$, let $S_i$ be the set of $e\in \binom{[r]}{\le s}$ such that there exist distinct $F,F'\in E_i$ with $c(F\cap F') \supset e$. Note that $S_0\supset S_1\supset \dots \supset S_\tau$, and $|S_0|\le T$. We will ensure that:
\begin{enumerate}
    \item\label{goodstop} for all $e\in S_\tau$ and $F\in E_\tau$, there exist distinct $F_1,\dots,F_t\in E_\tau$ such that $c(F\cap F_i) \supset e$ for each $i\in [t]$;
    \item\label{dense} for $i<\tau$, we have $|E_{i+1}|\ge (2tT^2)^{-1}|E_i|$;
    \item\label{shrink} for $i<\tau-1$, we have $|S_{i+1}|<|S_i|$.
\end{enumerate}By Items~\ref{shrink} and \ref{dense}, we see that $\tau\le T+1$ and thus $|E_\tau| \ge (2tT^2)^{-T-1}|E_0|$. Considering Item~\ref{goodstop}, we see that we can take $E^* = E_\tau$, and that the size condition follows because $|E_0|\ge \exp(-r)|E|$ and $2T^2\le 2^{2r+1}$.

It remains to describe our cleaning process and confirm that Items~\ref{dense}-\ref{shrink} hold.

Suppose we have defined $E_i$
Create a bipartite graph $\Gamma$ where  one part is $A:=E_i$, and the other part is $B:=\{F\cap c^{-1}(e):F\in E_i, e\in S_i\}$. For $F\in A, U\in B$, we say $F\sim U$ if $F\supset U$. Note that $d_\Gamma(F) = |S_t|\le T$ for $F\in A$. If $|B|\le  |A|/2tT$, then we can iteratively delete $\{U\}\cup N_\Gamma(U)$ for any $U$ with degree $<t$. At the end, we will be left with a graph $\Gamma'$ with $e(\Gamma')\ge e(\Gamma)/2$. We stop the process and take $E_\tau = A(\Gamma')$.

Otherwise, $|B|\ge |A|/(2tT)$. In which case, by pigeonhole, we can find $e_i\in S_i$ such such that $|B\cap c^{-1}(e_i)|\ge |B|/|S_i|\ge |A|/(2tT^2)$. Let $E_{i+1}\subset E_i$ to be a maximal set where $c^{-1}(e_i)\cap F\neq c^{-1}(e_i)\cap F'$ for distinct $F,F'\in E_{i+1}$, we have that $|E_{i+1}|= |B\cap c^{-1}(e_i)|\ge |E_i|/(2tT^2)$ and $e_i\not\in S_{i+1}$.\end{proof} 

\section{An extremal hypergraph problem}\label{hyperproblem}

We will need some terminology for hypergraphs, some of which is non-standard.

We say a hypergraph $\mathcal{H} = (V,E)$ is \textit{covered} if for every $v\in V$ there is some $e\in E$ where $v\in e$. Recall that we say $\mathcal{H}$ is {\em $\ell$-bounded} if $|e|\le \ell$ for $e\in E$.

For $A,B\subset V$, we say $(A,B)$ is an \textit{induced pair of order $k$} if $|B|=k$ and for each $b\in B$, there exists $e\in E$ such that $A\cup \{b\} \subset e$, but there is no $e\in E$ with $e\supset A$ and $|e\cap B| > 1$.
We define $\alpha(\mathcal{F})$ to be the maximum $k$ such that there exists an induced pair $(A,B)$ of order $k$.

Let $F(\ell,k)$ denote the minimum $n$ such that every covered $\ell$-bounded hypergraph $\mathcal{H}$ on at least $n$ vertices has $\alpha(\mathcal{H})\geq k$. The following lemma tells us this quantity is always finite. 

\begin{lemma}\label{hyper upper bound}For $\ell\ge 0$, we have that $F(\ell,k)\le \sum_{l=0}^\ell (k-1)^l$.
\end{lemma}

\begin{proof}  
We argue by induction on $\ell$. For $\ell = 0$, it is vacuously true that $F(0,k) = 1$, as any covered hypergraph $\mathcal{H}$ with $|V(\mathcal{H})|>0$ has an edge $e$ with $|e|>0$ (and hence is not $0$-bounded).  So we may assume $\ell\ge1$.

Given a covered hypergraph $\mathcal{H}=(V,E)$, we consider the graph $G = G_\mathcal{H}$ with $V(G) = V(\HH)$ where $u,v$ are adjacent if there is some $e\in E(\mathcal{H})$ containing both $u$ and $v$. We note that if $B\subset V$ is an independent set in $G$, then $(\emptyset,B)$ is an induced pair in $\mathcal{H}$. Indeed, since $\mathcal{H}$ is covered, every vertex of   
$B$ is contained in some edge of $\HH$; and no edge meets $B$ in more than one element, as $B$ is an independent set in $G$.  Thus $\alpha(\mathcal H)\ge|B|$

Now suppose $\alpha(\mathcal{H})<k$. Let $I\subset V$ be a maximal independent set in $G = G_{\mathcal{H}}$. We must have $|I|\le k-1$, and so there must be some $x\in I$ such that $d_G(x) \ge \frac{|V|-|I|}{|I|} \ge \frac{|V|}{k-1}-1$. 

Let $\mathcal{H}' = (V',E')$, where $V' = N_G(x)$ and $E' = \{e \cap V':e\in E,\,x\in e\}$.
Then $\alpha(\mathcal{H}')\le \alpha(\mathcal{H})$, as if $(A,B)$ is an induced pair of $\mathcal{H}'$, then $(A\cup \{i\},B)$ is an induced pair of $\mathcal{H}$. Also, 
$\mathcal{H}'$ is covered and $\max\{|e'|:e'\in E'\} < \max\{|e|:e\in E\} \le \ell$.
So $|V'| < F(\ell-1,k)$, and the result follows by induction.\end{proof}

It appears to be an interesting problem to understand how $F(\ell,k)$ grows. We comment on a few different regimes below.

Letting $R(K_a,K_b)$ denote the Ramsey number for a clique of size $a$ or independent set of size $b$, we observe that $F(\ell,k) \ge R(K_{\ell+1},K_k)$. Indeed, given a graph $G$ with no cliques of size $\ell+1$ or independent sets of size $k$, we can define $\mathcal{H}_G = (V(G),\{e\subset V(G):G[e] \textrm{ is  a clique}\})$, which is covered, $\ell$-bounded, and has $\alpha(\mathcal{H}_G)<k$.
It follows that $F(k,k) \ge R(K_k,K_k)\ge \exp(\Omega(k))$. Meanwhile, Theorem~\ref{hyper upper bound} tells us $F(k,k)\le k^k$. So in this regime, our upper bound (roughly) has the right asymptotic shape.

When $k=2$, Theorem~\ref{hyper upper bound} tells us that $F(\ell,2)\le \ell+1$. It is easy to see that equality holds here (consider a hypergraph with $\ell$ vertices, which are all contained in one edge).
However, for fixed $k\ge 3$, Theorem~\ref{hyper upper bound} only gives an exponential upper bound (i.e., that $F(\ell,k)\le \exp(O_k(\ell))$). Here it seems that a much better upper bound should hold. In fact, we conjecture $F(\ell,k) \le \ell^{O_k(1)}$. But already for $k=3$, we do not know how to prove $F(\ell,3)\le \exp(o(\ell))$.

\section{Proof of main theorem}\label{sec:mainproof}

We begin in the first subsection by handling a `model' case in Proposition~\ref{model lopsided}; in the following subsection, we prove Theorem~\ref{thm:main} by using the various cleaning lemmas from Section~\ref{sec:prelems} to reduce to the model case covered  by Proposition~\ref{model lopsided}

\subsection{The model case}

Consider $s,k\ge 2$. We want to argue that if $G$ is a $K_{s,s}$-free graph with large average degree, that it contains a $C_4$-free induced subgraph $G'$ with $d(G')\ge k$.

In our proof, all the main difficulties arise when $G = (A,B,E)$ is a bipartite graph that is `lopsided', meaning $|A|/|B|$ is very large compared to $d(G)$. By applying Lemmas~\ref{kernels} and \ref{hyper upper bound}, we shall prove the following `model case'.

\begin{proposition}\label{model lopsided}
    The exists some absolute constant $C$ so that the following holds for all $s,k\ge 2$ and $r\ge k^{Cs}$, and $|A|\ge 2^{r^{Cs}}|B|$. 
    
    Let $G= (A,B,E)$ be bipartite graph, with $d_G(a)= r$ for each $a\in A$.     
    Then either $G$ contains a $K_{s,s}$, or an induced subgraph $G'$, which is $C_4$-free and has $d(G')\ge k$.
\end{proposition}

\begin{proof}
We may assume $G$ is $K_{s,s}$-free.  Let $\rho:= |A|/|B|\ge 2^{r^{Cs}}$.

Consider the $r$-uniform hypergraph $\mathcal{F}$, where $V(\mathcal{F})=B$ and $e\in E(\mathcal{F})$ if there is $x\in A$ such that $N_{G}(x)=e$. Since $G$ is $K_{s,s}$-free and $r\ge s$, we have that $|E(\mathcal{F})| \ge |A|/s$ (because we can't have $s$ different vertices in $A$ corresponding to the same hyperedge of $\mathcal{F}$); very crudely we have
\[|E(\mathcal{F})|\ge |A|/s = (\rho/s)|V(\mathcal{F})| \ge \rho^{1/2}|V(\mathcal{F})|.\] Fix an injection $\phi:E(\mathcal{F})\to A$ so that $f = N_{G}(\phi(f))$ for each $f\in E(\mathcal{F})$. 

We apply Theorem~\ref{kernels} with parameters $r^*=r$, $t^*=\max\{k,s\}$ and $s^*=s$ to get an $r$-partite sub-hypergraph $\mathcal{F}'\subset \mathcal{F}$ satisfying the required properties and with $|E(\mathcal{F}')| \ge (t^*r2^{2r+1})^{-\sum_{i=0}^s\binom{r}{i} -1}|E(\mathcal{F})|$; very crudely we have \[t^*r2^{2r+1}\le 2^{3r}\quad \text{and}\quad \sum_{i=0}^s \binom{r}{s} +1\le r^s,\]
and so
\[|E(\mathcal{F}')|\ge 2^{-r^{s+3}}|E(\mathcal{F})| \ge \rho^{1/4}|V(\mathcal{F})| =\rho^{1/4}|V(\mathcal{F}')|.\]
Note that, as $\rho>1$, this implies
\[|E(\mathcal{F}')|>|V(\mathcal{F}'|.\]
Let $c:B\to [r]$ be the $r$-partite coloring associated with $\mathcal{F}'$, and $\mathcal{H}$ be the hypergraph on $[r]$ as in the last bullet of Theorem~\ref{kernels}.

By Lemma~\ref{hyper upper bound}, we have $|V(\mathcal{H})| = r \ge k^{Cs} \ge F(s-1,k)$. If $E(\mathcal{H})$ has an edge of size $s$, then $G$ contains a $K_{s,s}$ which is a contradiction. Thus $\mathcal{H}$ is $(s-1)$-bounded. It is also clear that $\mathcal{H}$ is covered, as $|E(\mathcal{F}')|>|V(\mathcal{F}')|$ (Indeed, consider $i\in [r] = V(\mathcal{H})$. If the singleton $\{i\}$ doe not belong to $E(\mathcal{H})$ then no vertex in the color class $c^{-1}(i)$ is contained in two edges $f,f'\in E(\mathcal{F}')$, by definition of $\mathcal{H}$. Recalling each $f\in E(\mathcal{F}')$ intersects the color class $c^{-1}(i)$, and noting $|c^{-1}(i)|\le |V(\mathcal{F}')|<|E(\mathcal{F}')|$, we get a contradiction due to pigeonhole.)

By definition of $F(s-1,k)$, $\mathcal{H}$ must have an induced pair of order $k$, say $(X,Y)$.

Fix some $f\in E(\mathcal{F}')$, and take $S\coloneqq f\cap c^{-1}(X)$ (i.e., its intersection with the parts corresponding to $X$). Now, we consider the common neighborhood \[A':=\phi(E(\mathcal{F}'))\cap \bigcap_{v\in S} N_{G}(v),\]which is nonempty (as it contains $\phi(f)$). Next, take \[B' := c^{-1}(Y) \cap \bigcup_{a\in A'} N_{G}(a).\]Finally, let $G'$ be the induced subgraph $G[A'\cup B']$.
\begin{claim}
$G'$ is $C_4$-free and $d(G')\geq k$.
\end{claim}
\begin{proof}
Suppose $G'$ contained a $C_4$. \hide{We claim this will contradict the assumption that $(X,Y)$ is an induced pair, meaning $G'$ must be $C_4$-free as desired. To see this, one notes that any such $C_4$ would have vertices $a_1,a_2\in A''$ and $b_1,b_2\in B'$ (as $G'$ is bipartite).}This implies\hide{ \alex{deleted footnote}}
%
that there are distinct $f_1,f_2 \in E(\mathcal{F}')$ such that $c(f_1\cap f_2)\supset X\cup \{y_1,y_2\}$ for some distinct $y_1,y_2\in Y$, contradicting the fact that $(X,Y)$ is an induced pair. It follows that $G'$ must be $C_4$-free as desired.

Meanwhile, for every $b\in B'$, there are at least $t\ge k$ hyperedges $f_1,\dots,f_t\in \mathcal{F}'$ containing $b$ and $S$ (by the definitions of $\mathcal{F}',\mathcal{H}$, and the fact that $(X,Y)$ is an induced pair) so $b$ has degree at least $k$ within $G'$. Also, for $a\in A'$, we have that $N_{G'}(a) = N_{G}(a)\cap c^{-1}(Y)$ implying $d_{G'}(a) = k$ (as $|Y|=k$). Thus $d(G')\ge k$ as desired.
\end{proof}
We see that $G'$ has the desired properties, completing the proof.
\end{proof}
\begin{remark}
    Inspecting the above proof, one can get the same conclusion as long as $r\ge F(s-1,k)$, and $\rho:= |A|/|B|$ satisfies $\rho > s k^{r^{s+3}}$.
\end{remark}

\subsection{The details for general graphs}
\begin{proof}[\unskip\nopunct]
Let $G$ be a graph with average degree $\dD \ge k^{Cs^3}$ (where $C$ is some large absolute constant). We assume $G$ is $K_{s,s}$-free, and wish to deduce that $G$ contains an induced $C_4$-free subgraph, $G'$, with $d(G') \ge k$. 
We may assume that $G$ is $\dD$-degenerate, or else pass to some induced subgraph $G^*$ of $G$ with larger average degree $d^*>\dD\ge k^{Cs^3}$. Thus the assumptions of Lemma~\ref{make extreme} are satisfied, as the degeneracy of $G$ is at least the average degree of $G$.

Let $\delta_0 := 1/200s$, and for $i>0$ set $\delta_i := \delta_0/2^i$. These quantities will be used throughout the proof.

We now apply Lemma~\ref{make extreme} with parameters $\dD := \dD,\delta:= \delta_0$.  There are two cases. 

\textbf{Case 1:} {\em $G$ contains an induced subgraph $G^*$ with $d(G^*)\ge 6d^{1-5\delta_0}$ and $\Delta(G^*)\le 6d^{1+3\delta_0}$}.  In this case, we use Lemma~\ref{deletion}. Indeed, we have $d(G^*)\ge \dD^{1/2}$ and $d(G^*) \ge \Delta(G^*)^{1-8\delta_0}= \Delta(G^*)^{1-1/25s}$. Assuming $C$ is large, we are guaranteed that $d(G^*) \ge \kappa^{25s}$ with room to spare (where $\kappa$ is the absolute constant from Lemma~\ref{deletion}).

We apply Lemma~\ref{deletion} with parameters $\dD:= d(G^*),s:= s$ and $\delta:= 1/25$ (note that $d(G^*) \ge \kappa^{25s}$, assuming $C$ is large). We obtain an induced subgraph $G'$ of $G$, which is $C_4$-free, and has average degree $d(G') \ge d(G^*)^{1/10s}\ge d(G)^{1/20s} = k^{(C/20)s^2}\ge k$.

\textbf{Case 2:} {\em There is a partition $V(G) = A_0 \cup B$ such that $e(G[A_0,B]) \ge nd/4$ and $|A_0| \ge 2^{d^{\delta_1}}|B|$}. Here, we will pass to some induced subgraph where we can apply Proposition~\ref{model lopsided}. 

Let $C_0$ be the absolute constant from that proposition. We assume that $C$ is large enough so that $C/1600 \ge C_0^2$ (note that $1/1600 = s\delta_1/4 \le s\delta_2$). 
For any $r \le d^{\min\{\delta_1/4,1/3s\}} = d^{\delta_1/4}$, we can apply Lemma~\ref{make bipartite} to $G[A_0\cup B]$ to obtain $A'\subset A_0$ and $B'\subset B$, where:
\begin{itemize}
    \item $G[A'],G[B']$ are both independent sets;
    \item $|A'|\geq 2^{\dD^{\delta_2}}|B'| \ge 2^{k^{C'^2s^2 }}|B|$ (recalling $d\ge k^{Cs^3}$); and
    \item $|N(a)\cap B'| = r$ for $a\in A'$.
\end{itemize}
Because $C$ is sufficiently large, we can apply this with $r := k^{C's}$ and obtain $A'\subset A_0$ and $B'\subset B$ such that $|A'| \ge 2^{r^{C's}}|B'|$.
Writing $G^* := G[A'\cup B']$, we can now apply Proposition~\ref{model lopsided} to find the desired induced subgraph $G'$ inside $G^*$. This completes the proof.
\end{proof}

Inspecting our proof, one obtains the following technical strengthening of Theorem~\ref{thm:main}.
\begin{theorem}\label{thm:technical}
    There exists an absolute constant $C$ such that the following holds for all $\dD,s,k \ge 2$ and $\delta\in (0,1/20)$ such that $\dD\ge (k/\delta)^{Cs^3/\delta}$.
    
    Let $G$ be a graph with average degree $d(G)\ge \dD$  without a $K_{s,s}$. Then $G$ has an induced subgraph $G'$ which is $\{C_3,C_4\}$-free with $d(G')\ge k$ and either
    \begin{itemize}
        \item $d(G')\ge \Delta(G')^{1-\delta}$; or
        \item $G'$ is bipartite.
    \end{itemize}
    Furthermore, we can ensure the first outcome if $\Delta(G)\le \dD 2^{\dD^{\delta/(1000s)}}$.
\end{theorem}
\begin{proof}
    Repeating the argument above with $\delta_0= \frac{\delta}{1000s}$, we will obtain $G'$ satisfying one of the two bullets. In Case 1, it is easy to check that the first outcome holds, by looking at the full statement of Lemma~\ref{deletion}. And in Case 2, we end up passing to a bipartite graph, causing the second outcome to hold.

    Under the assumption that $\Delta(G)\le \dD 2^{\dD^{\delta/(1000s)}}$, the set $R$ from the proof of Lemma~\ref{make extreme} will be empty, so we can find an induced subgraph $G^*$ with $d(G^*)\ge 6d^{1-5\delta_0},\Delta(G^*)\le 6d^{1+3\delta_0}$. In this case, we reach Case~1, and thus can ensure the first outcome.
\end{proof}

\section{Lower bounds}\label{sec:lower}

We require the following well-known extremal bound of Reiman \cite{reiman} (which slightly improved the bounds of K\H{o}v\'ari-S\'os-Tur\'an \cite{kovari}).

\begin{proposition}[{\cite[Equation~1.4]{reiman}}]\label{extremal number of c4}Any $n$-vertex $C_4$-free graph $G$ has at most $\frac{1}{2}n^{3/2}+n/4+1$ edges. \end{proposition}
\begin{corollary}\label{verts in C4free}Any $C_4$-free graph $G$ with average degree $d(G)\ge k\ge 1$ must have $|V(G)| \ge (k-1)^2$.\end{corollary}

Given an integer $K$ and $p\in [0,1]$, let $q(K,p)$ be the probability that $G\sim G(K,p)$ is $C_4$-free. (We write $G(K,p)$ for the Erd\H{o}s-R\'{e}nyi random graph where each edge of the $K$-vertex complete graph is included independently at random with probability $p$.)
\begin{proposition}\label{qbound}We have that \[q(K,p)\le (1-p^4)^{\binom{\lfloor K/2\rfloor}{2}} \le \exp\left(-p^4\binom{\lfloor K/2\rfloor}{2}\right).\]
\end{proposition}
\begin{proof}Sample $G\sim G(K,p)$. For $i<j\in [\lfloor K/2\rfloor ]$, $\mathcal{E}_{i,j}$ be the everealsednt that $G[\{2i-1,2i\},\{2j-1,2j\}]$ is a $C_4$.  Then $\mathbb{P}(\mathcal{E}_{i,j}) = 1-p^4\le \exp(-p^4)$, and these events are independent. The result follows.\end{proof}

\begin{theorem}\label{lower bound conditions}Let $k,s\ge 2$. Choose $p,n$ such that \[\max\{n^2p^s,n (1- p^4)^{(k^2-3k-2)/4},\exp(-\binom{n}{2}p/4)\}\le 1/2.\]Then there exists an $n$-vertex graph $G$ with $d(G)\ge \frac{n-1}{2}p$, such that $G$ is $K_{s,s}$-free, and every $C_4$-free induced subgraph $G'\subset G$ has $d(G')<k$.

\end{theorem}
\begin{proof} Consider $G\sim G(n,p)$
and fix $K:= k^2-3k$ (which is always even). Let $X$ count the number of $K$-subsets $S \in \binom{[n]}{K}$ such that $G[S]$ is $C_4$-free. By Corollary~\ref{verts in C4free}, if $X=0$, then $G$ does not contain an induced $C_4$-free subgraph $G'$ with $d(G')\ge k$. We note that (by Proposition~\ref{qbound})
\begin{align*}
    \mathbb{E}[X] = \binom{n}{K} q(K,p)&\le n^K (1-p^4)^{\binom{K/2}{2}} 
    = \left(n (1-p^4)^{(k^2-3k-2)/4}\right)^K.
\end{align*}
Now let $Y$ count the number of copies of $K_{s,s}$ in $G$. Then $\E[Y] \le n^{2s}p^{s^2} = (n^2p^s)^s$.

By assumption, $k,s\ge 2$ (and so $K\ge 2$). Since $\max\{n(1-p^4)^{(k^2-3k-2)/4},n^2p^s\}\le 1/2$, we get \[\mathbb{P}(X=Y=0) \ge 1- \mathbb{E}[X]-\mathbb{E}[Y]\ge 1/2.\]Meanwhile, a Chernoff bound tells us that $\mathbb{P}(e(G)\le \binom{n}{2}p/2) < \exp(-\binom{n}{2} p/4)$.

It follows by our assumptions on $n,p$ that with positive probability, we will have $X=Y=0$ realsedrealsedrealsedand $e(G)\ge p\binom n2/2$. Taking such a $G$ gives our result.
\end{proof}

Recall that $f_{\Ind}(s,k)$ denotes the smallest $D$ such that for any $K_{s,s}$-free graph $G$ with $d(G)\ge D$, there exists a $C_4$-free induced subgraph $G'\subset G$ with $d(G)\ge k$. We will use Theorem~\ref{lower bound conditions} to get some lower bounds for this function.
\begin{corollary}[Theorem~\ref{thm:lower} \ref{diagonal case}]\label{diagonal lower}
   For sufficiently large $k$, we have that $f_{\Ind}(k,k)\ge k^{k/21}$.
\end{corollary}
\begin{proof}
    Assuming $k$ is sufficiently large, we may apply Theorem~\ref{lower bound conditions} with $n = k^{k/20}$ and $p =k^{-1/5}$. 
    
    It is trivial to verify that the first and third terms in the statement of Theorem~\ref{lower bound conditions} are appropriately bounded, so we only discuss the second condition. Here (for $k\ge 10$), the crude bound $(1-p^4)^{(k^2-3k-2)/4}\le \exp(-p^4 (k^2-3k-2)/4)\le \exp(-k^{6/5}/8)$ gives us what we need (as $\log n\le k\log k\le k^{6/5}/8-1$ for large $k$).
\end{proof}
\begin{remark}
    It is not hard to prove that the probabilistic constructions in this section are almost-regular with high probability. Thus, Corollary~\ref{diagonal lower} (along with Corollary~\ref{fixed s lower}, detailed later below) tell us that Lemma~\ref{deletion} is essentially tight (up to the constant in our exponent) when we are looking for subgraphs with large average degree. This is rather surprising since in the non-induced setting one gets much better (i.e., polynomial) bounds when $G$ is almost-regular. 
\end{remark}

\begin{corollary}[Theorem~\ref{thm:lower} \ref{fixed k case}] \label{fixed k lower}
    Fix $k\ge 2$. We have that $f_{\Ind}(s,k)\ge s^{(1/4-o(1))(k^2-3k-2)}$.
\end{corollary}
\begin{proof} Note that this result is trivial for $k=2,3$, as then $k^2-3k-2<0$. So now assume $k\ge 4$.

We use Theorem~\ref{lower bound conditions} again, with $n = s^{(1/4-\epsilon)(k^2-3k-2)},p = 1-k^2\log s/s$ where $\epsilon = \epsilon(s)$ tends to zero as $s\to \infty$.
Crudely, we have that $n^2p^s \le \exp(\log s(k^2-3k-2)-k^2\log s)<1/2$. Meanwhile, we have that $(1-p^4)^{(k^2-3k-2)/4} \le (4k^2\log s/s)^{(k^2-3k-2)/4} \le s^{(o(1)-1/4)(k^2-3k-2)}<1/2n$ (assuming $\epsilon(s)$ does not tend to zero too fast). The third condition holds as $s\to\infty$, because $n\to \infty$ and $p\to 1$.
\end{proof}

Our final lower bound does not require Theorem~\ref{lower bound conditions}. Instead, we need a basic observation. Here we write $\alpha(G)$ for the maximum size of an independent set of a graph $G$.
\begin{proposition}
    Any $n$-vertex $C_4$-free graph $G$ has an independent set of size $\ge \frac{n}{3+\sqrt{n}}$.
\end{proposition}
\begin{proof}
    This follows from Proposition~\ref{extremal number of c4}, and the fact that $\alpha(G)\ge \frac{|V(G)|}{1+d(G)}$ for any graph $G$ (a well-known corollary of Tur\'an's Theorem).
\end{proof}
Combined with Corollary~\ref{verts in C4free}, we get the following.
\begin{corollary}\label{independent set in C4free}
    Any $C_4$-free graph $G$ with average degree $d(G)\ge k\ge 1$ must have $\alpha(G)\ge \frac{(k-1)^2}{k+2}\ge k-4$.
\end{corollary}
\begin{corollary}[Theorem~\ref{thm:lower} \ref{fixed s case}]\label{fixed s lower} Fix $s\ge 2$. Then $f_{\Ind}(s,k) \ge k^{(1/2-o(1))s-1}$.
    
\end{corollary}
\begin{proof}
    Let $k' = k-4$. We consider $G\sim G(n,p)$, where $n,p$ will be chosen later.    
    Let $X$ count the number of $k'$-subsets $S\in \binom{[n]}{k'}$ such that $G[S]$ is an independent set. By Corollary~\ref{independent set in C4free}, if $X=0$, then $G$ will not contain an induced $C_4$-free subgraph $G'$ with $d(G')\ge k$. We note that $\E[X] = \binom{n}{k'}(1-p)^{\binom{k'}{2}} \le (n(1-p)^{k/4})^{k'}$ (assuming $k\ge 10$).

    Let $Y$ count the number of copies of $K_{s,s}$ in $G$. We have $\E[Y]\le n^{2s}p^{s^2}\le (n^2p^s)^{s}$.

    Finally, let $\mathcal{E}$ be the event that $e(G)\le \binom{n}{2}p/2$. A Chernoff bound tells us $\P(\mathcal{E})<\exp(-\binom{n}{2}p/4)$.

    If $\E[X]+\E[Y]+\P(\mathcal{E})<1$, a union bound tells us that $f_{\Ind}(s,k)\ge \frac{(n-1)p}{2}$. Taking $p = n^{-(2+\eps)/s}$, $n=k^{\frac{1-\eps}{2+\eps}s}$ so that $p=k^{\eps-1}$, one can check that this holds for sufficiently large $k$.
\end{proof}

\begin{remark} To provide lower bounds for $f_{\Ind}(s,k)$, we found $K_{s,s}$-free graphs $G$ with the stronger property that every set of $(k-1)^2$ vertices contains a $C_4$. We note that one cannot hope to do better with such an approach (beyond improving the constants in the exponents).
Indeed, such $G$ must not contain a clique on $2s$ vertices, nor an independent set on $(k-1)^2$ vertices. Thus, the average degree of $G$ (which is at most $|V(G)|-1$), will be bounded by the off-diagonal Ramsey number $R(K_{2s},K_{k^2})$. A classical bound of Erd\H{o}s-Szekeres \cite{ES} tells us \[R(K_a,K_b)\le \binom{a+b-2}{b-1}\le (a+b)^{\min\{a,b\}},\] and in particular we get
\[R(K_{2k},K_{k^2}) \le k^{(4+o(1))k}\quad\textrm{and}\quad R(K_{2s},K_{k^2})\le \min\{s^{(1+o(1))k^2},k^{(4+o(1))s}\}.\]
\end{remark}

\section{Corollary on subdivisions}\label{finding subdivisions}

In this section we prove Corollary~\ref{induced subdivision}, which improves the bounds 
in the theorem of K\"uhn and Osthus ~\cite{KO2} (Theorem \ref{kotheorem}) from triply exponential to singly exponential.  We restate the result for convenience.

\indSub*

We require the following result, which was proved independently 
by Bollob\'as and Thomason~\cite{bollobas} and K\'omlos and Szemer\'edi~\cite{komlos}. 
\begin{theorem}\label{noninduced subdivision}
    There exists an absolute constant $C>0$ such that for all $k \in \N$, if $G$ is a graph with $d(G)\ge Ck^2$, then $G$ contains a (not necessarily induced) subdivision of $K_k$.
\end{theorem}


We can now proceed with the proof of the corollary.

\begin{proof}[Proof of Corollary~\ref{induced subdivision}]
Let $G$ be a $K_{s,s}$-free graph with $d(G)\ge k^{Cs^3}$ (where $C$ is a large constant). Applying Theorem~\ref{thm:technical} with $\delta = 1/100$, we can find an induced $\{C_3,C_4\}$-free subgraph $G'$ with average degree $C'k^5$ (with $C' =\Omega(C)$), where either $G'$ is bipartite or $d(G')\ge \Delta(G')^{1-1/100}$. We handle these in separate cases.

\textbf{Case 1} {\em $G'$ is bipartite.} We pass to an induced subgraph $H\subset G'$ with maximum average degree, say $\dD$.  Then $\dD\geq C'k^2$ and $H$ is $\dD$-degenerate with $\delta(H)\ge \dD/2$. Let $H=(A,B)$ where $|A'|\geq |B'|$. Let $A'=\{ x\in A: d_H(x)\geq 4\dD\}$. 

We have $|A'|\leq |A|/2 $ (by counting edges), and by maximality of average degree, we see that $e(H[A',B])\le (3/4)e(H)$. 
    
    Now, let $A^* = A\setminus A'$ and $F=G[A^*, B]$. We note $e(F)\ge e(H)/4$ and $|A^*|\ge |B|/2$. Furthermore, as $\delta(H)\ge d/2$, we have $d_F(a) = d_{H}(a)\in [\dD/2,4\dD]$ for all $a\in A^*$ 
    
    Let $W$ be a random subset of $B$, where each vertex is included independently in $W$ with probability $p=1/8\dD$. Let $U = \{x\in A^*:|N_F(x)\cap W| =2\}$. For $x\in A^*$, we have
    \begin{align*}
        \mathbb{P}[x\in U] = p^2(1-p)^{d_F(x)-2} \binom{d_F(x)}{2}
        \ge p^2(1-p)^{4\dD}\binom{\dD/2}{2}
        \ge p^2\dD^2/20\ge 1/2000.
    \end{align*}

    Then $\mathbb{E}[|U|-\dD|W|/1000]\geq |A^*|/2000-|B|/8000> 0$ (as $|A^*|\ge |B|/2$) and hence there is a choice of $W$ such that $|U|> \dD|W|/1000$. Fix such a $W$ and define an auxiliary graph $J$ with vertex set $W$ and $E(J) = \{N(z):z\in U\}$. 
    
    Note that $e(J) = |U|$, as $F$ was $C_4$-free (and so the edges coming from different $z$ are distinct). Thus, we have $d(J)= C''k^2$ for some $C'' \ge C'/1000$. Taking $C$ (and thus $C''$) sufficiently large, it follows that $J$ must contain a subdivision of a $K_{k}$ by Theoren~\ref{noninduced subdivision}.  By construction, this corresponds to an induced subdivision of $K_k$ in $G$.  

    \textbf{Case 2} {\em $d(G')\ge \Delta(G')^{1-1/100}$.}  As in the first case, it is enough to find $U,W\subset V(G')$ such that:
    \begin{itemize}
        \item $U$ and $W$ are independent sets; 
        \item $d_{G'}(u) = 2$ for each $u\in U$; and
        \item $|U|\ge C''k^2|W|$ for some appropriately large constant $C''>0$.
    \end{itemize} Theorem~\ref{noninduced subdivision} then gives our induced subdivision of $K_k$, as desired.

    We first pass to a subgraph $H\subset G'$ of maximal average degree.  Then $\dD:=d(H)\ge \max\{C'k^5,\Delta(H)^{1-1/100}\}$ and also $\delta(H)\ge \dD/2$.  Let $W_0$ be a random subset of $V(H)$ where each vertex $v\in V(H)$ is included in $W_0$ with probability $p =1/10\dD^{8/5}$. Let $W\subset W_0$ be the set $\{w\in W_0:|N_H(w)\cap W_0|=0\}$. Finally define $U_0:= \{u\in V(H):|N_H(u)\cap W|=|N_H(u)\cap W_0|=2\}$ and let $U\subset U_0\setminus W_0$ be the set of $u\in U_0$ with $u\not\in W_0$ and $|N_H(u)\cap U_0| = 0$. 
    
     So for $x\in V(G)$, 
     a union bound tells us
    \[\P(x\in U) \ge \P(x\not\in W_0)\P(x\in U_0|x\not\in W_0)\left(1-\sum_{y\in N_H(x)}\P(y \in U_0|x\in U_0)\right),\]
    where we have used the fact that if $x\in U_0$ then $x\not\in W_0$.
    
    Now $\P(x\not\in W_0) = 1-p\ge 1/2$, and
    \[\P(x\in U_0|x\not\in W_0) = \sum_{\{y,y'\}\in \binom{N_H(x)}{2}}\P(W_0\cap N_H(x) = W\cap N_H(x) = \{y,y'\}).\]Now, for each $\{y,y'\}\in \binom{N_H(x)}{2}$, we have $y,y'$ are not adjacent because $H\subset G'$ is $C_3$-free (and $x\in N_H(y)\cap N_H(y')$), thus the corresponding summand is at least
    \[(1-p)^{d_H(x)-2}p^2 (1-\E[|W_0\cap N_H(y)|+|W_0\cap N_H(y')|])\ge p^2(1-p\Delta(H))(1-2p\Delta(H))\ge p^2/2.\]Consequently, $\P(x\in U_0)\ge \binom{d_H(x)}{2} p^2/2\ge \dD^2p^2/10$ (since $\delta(H)\ge \dD/2$).

    Finally, for $y\in N_H(x)$, we have $N_H(y)\cap N_H(x) = \emptyset$ because $H\subset G'$ is $C_3$-free, whence
    \begin{align*}
        \P(y\in U_0|x\in U_0\setminus W_0) &= \P(y\in U_0|x\not \in W_0)\\
        &\le  p^2\binom{d_H(y)-1}{2}\\
        &\le (p\Delta(H))^2.\\
    \end{align*}

    Putting these together, we get
    \[\P(x\in U)\ge \frac{p^2\dD^2}{10}(1/2)(1-\Delta(H)(p\Delta(H))^2) \ge \frac{p^2\dD^2}{40}\](where in the last step we use $2(1-1/100)8/5>3$ to deduce $p^2\Delta(H)^3\le 1/2$). 
    
    Consequently, $\E[|U|]\ge \frac{p\dD^2}{40}\E[|W_0|] \ge \frac{\dD^{2/5}}{40}\E[|W|]$. Thus, there is some choice of $W_0$ such that $|U|\ge C''k^2|W|$ (where $C'' = \Omega(C'^{2/5})$).  Taking $C$ (and thus $C''$) sufficiently large, we are done.
\end{proof}
\section{Concluding remarks}\label{sec:conclude}
    \begin{itemize}
        \item It is still very interesting to improve the bounds in the non-induced case. Let $f(k,6)$ denotes the least integer $\dD$ such that if $G$ is a graph with $d(G)\ge \dD$, then $G$ contains a (not-necessarily induced) $C_4$-free subgraph $G'\subset G$ with $d(G)\ge k$.

        The best known bounds are  
        \[k^{3-o(1)}\le f(k,6)\le k^{O(k^2)},\]which were both established in \cite{MPS}. It would be very interesting if one could show that a polynomial upper bound held (i.e., that $f(k,6)\le k^{O(1)}$).
    
        \item We believe that for every $k$, there is a polynomial $p_k(s)$ so that every $K_{s,s}$-free graph $G$ with $d(G)\geq p_k(s)$ contains an induced $C_4$-free subgraph with average degree at least $k$. We could not verify this conjecture, but if true this would show that every {degree-bounded} hereditary family of graphs is polynomially bounded. We note that recently Bria\'nski, Davies, and Walczak~\cite{non-polychi} showed that there are $\chi$-bounded hereditary families of graphs whose chromatic number can grow arbitrarily fast compared with the clique number.

        \item An interesting difference between the two problems above is that, in the induced case, things are still difficult when $G$ is almost-regular. Let $f'_{\Ind}(s,k)$ be the smallest integer $d_0$ such that every $K_{s,s}$-free graph $G$ so that $d(G)/2 \le \delta(G) \le \Delta(G) \leq 2d(G)$
        contains an induced $C_4$-free subgraph $G'\subset G$ with $d(G')\ge k$.

        Applying Lemma~\ref{deletion}, we see that $f_{\Ind}'(s,k)\le k^{Cs}$ for some absolute constant $C$. It would already be nice to get a polynomial bound here (i.e., prove $f_{\Ind}'(s,k)= s^{O_k(1)}$).

        \item Finally, it would be very nice to better understand the behaviour of the function $F(\ell,k)$ introduced in Section~\ref{hyperproblem}. Already for $k=3$ and large $\ell$, we only know
        \[\ell^{2-o(1)}\le F(\ell,3) \le O(2^\ell),\]and improving either of these bounds would be quite interesting. We tentatively expect a polynomial upper bound when $\ell = O(1)$. 
    \end{itemize}

\bibliographystyle{abbrv}
\bibliography{Inducedfourcycle}

\begin{thebibliography}{10}

\bibitem{bollobas}
B.~Bollob{\'a}s and A.~Thomason.
\newblock Proof of a conjecture of {M}ader, {E}rd{\H{o}}s and {H}ajnal on
  topological complete subgraphs.
\newblock {\em Europ. J. Comb.}, 19:883--887, 1998.

\bibitem{bonamy}
M.~Bonamy, N.~Bousquet, M.~Pilipczuk, P.~Rz\k{a}\.zewski, S.~Thomass\'e, and
  B.~Walczak.
\newblock Degeneracy of ${P}_t$-free and ${C}_{\ge t}$-free graphs with no
  large complete bipartite subgraphs.
\newblock {\em J. {C}ombin. {T}heory {S}er. {B}}, 152:353–378, 2022.

\bibitem{non-polychi}
M.~Bria\'nski, J.~Davies, and B.~Walczak.
\newblock Separating polynomial $\chi$-boundedness from $\chi$-boundedness.
\newblock {\em ArXiv:2201.08814v2}.

\bibitem{CHMS23}
A.~Carbonero, P.~Hompe, B.~Moore, and S.~Spirkl.
\newblock A counterexample to a conjecture about triangle-free induced
  subgraphs of graphs with large chromatic number.
\newblock {\em Journal of Combinatorial Theory, Series B}, 158:63--69, 2023.

\bibitem{CSS}
M.~Chudnovsky, A.~Scott, and P.~Seymour.
\newblock Induced subgraphs of graphs with large chromatic number. {III}.
  {L}ong holes.
\newblock {\em Combinatorica}, 37:1057--1072, 2017.

\bibitem{DelRodl}
D.~Dellamonica and V.~R\"odl.
\newblock A note on {T}homassen's conjecture.
\newblock {\em J. {C}ombin. {T}heory {S}er. {B}}, 101:509--515, 2011.

\bibitem{d11}
D.~Dellamonica~Jr, V.~Koubek, D.~M. Martin, and V.~R{\"o}dl.
\newblock On a conjecture of {T}homassen concerning subgraphs of large girth.
\newblock {\em Journal of Graph Theory}, 67(4):316--331, 2011.

\bibitem{ErdosSauer}
P.~Erd\H{o}s.
\newblock On the combinatorial problems which {I} would most like to see
  solved.
\newblock {\em Combinatorica}, 1:25--42, 1981.

\bibitem{ES}
P.~Erd{\H{o}}s and G.~Szekeres.
\newblock A combinatorial problem in geometry.
\newblock {\em Compositio Mathematica}, 2:463--470, 1935.

\bibitem{esperet2017habilitation}
L.~Esperet.
\newblock {\em Graph Colorings, Flows and Perfect Matchings}.
\newblock Habilitation thesis, Université Grenoble Alpes, 2017.

\bibitem{Furedi}
Z.~F\"uredi.
\newblock On finite set-systems whose every intersection is a kernel of a star.
\newblock {\em Discrete Math.}, 47:129--132, 1983.

\bibitem{inducedinduced}
A.~Gir{\~a}o, F.~Illingworth, E.~Powierski, M.~Savery, A.~Scott, Y.~Tamitegama,
  and J.~Tan.
\newblock Induced subgraphs of induced subgraphs of large chromatic number.
\newblock {\em ArXiv:2203.03612}.

\bibitem{G87}
A.~Gy\'{a}rf\'{a}s.
\newblock Problems from the world surrounding perfect graphs.
\newblock {\em Applicationes Mathematicae}, 19(3-4):413--441, 1987.

\bibitem{GST80}
A.~Gy\'{a}rf\'{a}s, E.~Szemeredi, and Z.~Tuza.
\newblock Induced subtrees in graphs of large chromatic number.
\newblock {\em Discrete Mathematics}, 30(3):235--244, 1980.

\bibitem{JS}
O.~Janzer and B.~Sudakov.
\newblock Resolution of the {E}rd{\H{o}}s-{S}auer problem on regular subgraphs.
\newblock {\em Forum of Mathematics, Pi}, 11:e19, 2023.

\bibitem{KP}
H.~A. Kierstead and S.~G. Penrice.
\newblock Radius two trees specify $\chi$-bounded classes.
\newblock {\em J. Graph Theory}, 18:119--129, 1994.

\bibitem{komlos}
J.~Koml\'os and E.~Szemerédi.
\newblock Topological cliques in graphs {II}.
\newblock {\em Combin., Probab. and Comput.}, 5:79--90, 1996.

\bibitem{KO1}
D.~K\"uhn and D.~Osthus.
\newblock Every graph with sufficiently large average degree contains a
  ${C}_4$-free subgraph of large average degree.
\newblock {\em Combinatorica}, 24:155--162, 2004.

\bibitem{KO2}
D.~K\"uhn and D.~Osthus.
\newblock Induced subdivisions in ${K}_{s,s}$-free graphs of large average
  degree.
\newblock {\em Combinatorica}, 24:287--304, 2004.

\bibitem{KLST}
M.~Kwan, S.~Letzter, B.~Sudakov, and T.~Tran.
\newblock Dense induced bipartite subgraphs in triangle-free graphs.
\newblock {\em Combinatorica}, 40:283--305, 2020.

\bibitem{kovari}
T.~Kóvari, V.~Sós, and P.~Turán.
\newblock On a problem of {K}. {Z}arankiewicz.
\newblock {\em Colloquium Mathematicae}, 3(1):50--57, 1954.

\bibitem{McCarty}
R.~McCarty.
\newblock Dense induced subgraphs of dense bipartite graphs.
\newblock {\em Siam J. Discrete Math.}, 35(2):661--667, 2021.

\bibitem{MPS}
R.~Montgomery, A.~Pokrovskiy, and B.~Sudakov.
\newblock ${C}_{4}$-free subgraphs with large average degree.
\newblock {\em Israel J. of Math.}, 246:55--71, 2021.

\bibitem{PRS}
L.~Pyber, V.~R\"odl, and E.~Szemer{\'e}di.
\newblock Dense graphs without $3$-regular subgraphs.
\newblock {\em J. Combin. Theory Ser. B}, 63:41--54, 1995.

\bibitem{reiman}
I.~Reiman.
\newblock Über ein problem von {K}. {Z}arankiewicz.
\newblock {\em Acta Mathematica Academiae Scientiarum Hungarica}, 9:269--273,
  1953.

\bibitem{inducedx}
A.~Scott and P.~Seymour.
\newblock Induced subgraphs of graphs with large chromatic number. x. holes of
  specific residue.
\newblock {\em Combinatorica}, 39(5):1105--1132, 2019.

\bibitem{SSSurvey}
A.~Scott and P.~Seymour.
\newblock A survey of $\chi$-boundedness.
\newblock {\em J. Graph Theory}, 95:473--504, 2020.

\bibitem{SSS}
A.~Scott, P.~Seymour, and S.~Spirkl.
\newblock Polynomial bounds for chromatic number. {I}. {E}xcluding a biclique
  and an induced tree.
\newblock {\em J. Graph Theory}, 102:458--471, 2023.

\bibitem{Thomassen}
C.~Thomassen.
\newblock Girth in graphs.
\newblock {\em J. {C}ombin. {T}heory {S}er. {B}}, 35:129--141, 1983.

\end{thebibliography}

\end{document}